\documentclass{amsart}
\usepackage{graphicx}
\usepackage{amscd}
\usepackage{amsmath}
\usepackage{amsfonts}
\usepackage{amssymb}

\newcommand{\setN}{\ensuremath{\mathbb{N}}}
\newcommand{\setZ}{\ensuremath{\mathbb{Z}}}
\newcommand{\setQ}{\ensuremath{\mathbb{Q}}}
\newcommand{\setR}{\ensuremath{\mathbb{R}}}

\newcommand{\setRext}{\ensuremath{\setR\cup\left\{\infty\right\}}}
\newcommand{\setF}{\ensuremath{\mathcal{F}}}

\newcommand{\setRscript}{\ensuremath{\mathcal{R}}}

\newcommand{\abs}[1]{\ensuremath{\left|#1\right|}}
\newcommand{\val}[1]{\ensuremath{\left\|#1\right\|}}
\newcommand{\genval}[3]{\ensuremath{\left\|#1\right\|_{\left(#2, #3\right)}}}
\newcommand{\gamval}[2]{\genval{#1}{\Gamma}{#2}}
\newcommand{\unval}[2]{\genval{#1}{u}{#2}}
\newcommand{\wkval}[2]{\genval{#1}{w}{#2}}

\newcommand{\Top}[1]{\ensuremath{\tau_{#1}}}
\newcommand{\topv}{\Top{v}}
\newcommand{\topu}{\Top{u}}
\newcommand{\topg}{\Top{\Gam}}
\newcommand{\topw}{\Top{w}}
\newcommand{\Gam}{\ensuremath{\Gamma}}
\newcommand{\Gamn}[1]{\ensuremath{\Gamma_{#1}}}
\newcommand{\gam}{\ensuremath{\gamma}}

\newcommand{\pb}[3]{\ensuremath{\PB_{#1}\left(#2, #3\right)}}
\newcommand{\pbg}[2]{\pb{\Gam}{#1}{#2}}

\newcommand{\pbu}[2]{\pb{u}{#1}{#2}}
\newcommand{\dg}[2]{\ensuremath{\Delta_{\Gamma}\left(#1, #2\right)}}
\newcommand{\Dg}{\ensuremath{\Delta_{\Gamma}}}

\newcommand{\Dw}{\ensuremath{\Delta_{w}}}
\newcommand{\du}[2]{\ensuremath{\Delta_{u}\left(#1, #2\right)}}
\newcommand{\Du}{\ensuremath{\Delta_{u}}}

\newtheorem{prp}{Proposition}[section]
\newtheorem{thm}[prp]{Theorem}
\newtheorem{dfn}[prp]{Definition}
\newtheorem{xpl}[prp]{Example}
\newtheorem{lmm}[prp]{Lemma}
\newtheorem{rmk}[prp]{Remark}
\newtheorem{crl}[prp]{Corollary}

\DeclareMathOperator{\supp}{supp}
\DeclareMathOperator{\PB}{PB}

\begin{document}
\title[Topologies on the Hahn field and convergence of power series]{On the topological structure of the Hahn field and convergence of power series}
\author{Darren Flynn}
\address{Department of Physics and Astronomy, University of Manitoba, Winnipeg, Manitoba
R3T 2N2, Canada}
\email{flynnd3@myumanitoba.ca}
\author{Khodr Shamseddine}
\address{Department of Physics and Astronomy, University of Manitoba, Winnipeg, Manitoba
R3T 2N2, Canada}
\email{Khodr.Shamseddine@umanitoba.ca}
\begin{abstract}
In this paper, we study the topological structure of the Hahn field $\mathcal{F}$ whose elements are functions from the additive abelian group of rational numbers $\setQ$ to the real numbers field $\setR$, with well-ordered support. After reviewing the algebraic and order structures of the field $\setF$, we introduce different vector topologies on $\mathcal{F}$ that are induced by families of semi-norms and all of which are weaker than the order or valuation topology. We compare those vector topologies and we identify the weakest one which we denote by $\topw$ and whose properties are similar to those of the weak topology on the Levi-Civita field \cite{rstopology09}. In particular, we state and prove a convergence criterion for power series in $\left(\mathcal{F},\topw\right)$ that is similar to that for power series on the Levi-Civita field in its weak topology \cite{sham}.
\end{abstract}
\maketitle

	\section{Introduction}
	Functions on non-Archimedean fields often display properties that appear very different from those of real-valued functions on the real field $\setR$. In particular it is possible to construct continuous functions that are not bounded on a closed interval, continuous and bounded functions that attain neither a maximum nor a minimum value on closed intervals, and continuous and differentiable functions with a derivative equal to zero everywhere on their domain which are nevertheless non-constant \cite{rsrevitaly13}. These unusual properties are a result of the total disconnectedness of these structures in the order topology \cite{shamberz, rsrevitaly13}. Much work has been done in showing that power series and analytic functions on the Levi-Civita field have the same smoothness properties as real power series and real analytic functions  \cite{sham}. The effort to extend these properties to as large a class of functions as possible has been aided considerably by the introduction of the so-called weak topology on the Levi-Civita field which is strictly weaker than the order topology and thus allows for more power series to be converge than the order topology. In this paper we will show how a similar weak topology may be induced on the Hahn field and we derive the convergence criterion for sequences and power series in this new topology. We begin with a brief introduction to the Hahn and Levi-Civita fields.
\begin{dfn}[Well-ordered subsets of $\setQ$]
	Let $A\subset\setQ$. Then we say that $A$ is \textbf{well-ordered} if every nonempty subset of $A$ has a minimum element.
\end{dfn}

\begin{dfn}[Left-finite subsets of $\setQ$]
	Let $A\subset\setQ$.  Then we say that $A$ is \textbf{left-finite} if, for any $q\in\setQ$, the set
	\[
		A_{<q} :=\left\{a\in A\middle|a<q\right\}
	\]
	is finite.
\end{dfn}

\begin{dfn}[The support of a function from $\setQ$ to $\setR$]
	Let $f:\setQ\rightarrow\setR$. The then the support of f is denoted by $\supp(f)$ and is defined by
	\[
		\supp(f) := \left\{q\in\setQ\middle|f(q) \ne 0\right\}.
	\]
\end{dfn}

\begin{dfn}[The sets $\setF$ and $\setRscript$]
	We define
	\[
		\setF := \left\{f:\setQ\rightarrow\setR\middle|\supp(f) \mbox{ is well-ordered}\right\}
	\]
	and
	\[
		\setRscript := \left\{f:\setQ\rightarrow\setR\middle|\supp(f) \mbox{ is left-finite}\right\}.
	\]
\end{dfn}

Note that every left-finite subset of $\setQ$ is also well-ordered but the converse is not necessarily true as there well-ordered subsets of $\setQ$ that are not left-finite. It follows that $\mathcal{R}\varsubsetneq\mathcal{F}$.

\begin{rmk}[Notation regarding elements versus functions]
	In the course of this document we will have opportunities to discuss both elements of $\setF$ and $\setRscript$ as well as functions on those sets. Since the elements of $\setF$ and $\setRscript$ are themselves functions from $\setQ$ to $\setR$ there is some possibility of confusion; for this reason we introduce the convention that square brackets (i.e. `$[$' and `$]$') denote an element of either $\setF$ or $\setRscript$ evaluated at some point in $\setQ$ where as curved brackets (i.e. `$($' and `$)$') denote a function on $\setF$ or $\setRscript$ evaluated at a point in one of those sets. So for example if we have $x\in\setF$, $q\in\setQ$, and $f:\setF\rightarrow\setF$, then
	\begin{itemize}
		\item $x[q]\in\setR$ denotes an element of $\setF$ evaluated at a point in $\setQ$. The result of the evaluation will of course be a real number.
		\item $f(x)\in\setF$ denotes a function evaluated at a point in $\setF$. The result of the evaluation is another element of $\setF$.
		\item $f(x)[q]\in\setR$ denotes a function evaluated at a point in $\setF$ and the result of that evaluation (itself an element of $\setF$) evaluated at a point in $\setQ$. Again, the result of this expression is a real number.
	\end{itemize}
\end{rmk}
We also introduce the following notation which will be useful when define the order later.
\begin{dfn}($\lambda$)
	Let $x\in \setF$ (resp. let $x\in\setRscript$). Then we define
	\[
		\lambda(x) := \begin{cases}\min{\supp(x)} &\mbox{ if } x\neq 0\\ \infty &\mbox{ if } x=0.\end{cases}
	\]
\end{dfn}

 Note that, for $x\ne 0$ in $\mathcal{F}$, $\min{\supp(x)}$ exists in $\setQ$ since the support is well-ordered. As we will see later after defining the order on $\mathcal{F}$ (resp. on $\mathcal{R}$), for $x\in\setF$, $\lambda(x)$ will correspond to the order of magnitude of $x$.

\begin{dfn}[addition and multiplication on $\setF$ and $\setRscript$]
	Let $x,y\in\setF$ (resp. $x,y\in\setRscript$) be given. Then we define $x+y$ and $x\cdot y$ as follows: for every $q\in\setQ$, let
	\begin{eqnarray*}
		(x + y)[q] &=& x[q] + y[q],\mbox{ and}\\
\mbox{ }&\mbox{ }&\mbox{ }\\
	(x\cdot y)[q] &=& \sum\limits_{\begin{tiny}\begin{array}{l}q_1 \in \supp(x)\\q_2 \in \supp(y)\\q_1 +q_2 =q \end{array}\end{tiny}}x[q_1]\cdot y[q_2]
	\end{eqnarray*}
\end{dfn}
 We have by \cite[Theorem 1.3]{graftonthesis} that if $A,B$ are well-ordered sets and $r\in A+B$ then there are only finitely many pairs $(a,b)\in A\times B$ such that $a+b = r$. This fact ensures that multiplication on $\setF$ (and $\setRscript$) is well-defined since the sum in the definition will always have finitely only many terms. Moreover, supp$(x+y)$ and supp$(x\cdot y)$ are well-ordered (resp. left-finite) in $\setQ$ and hence $\mathcal{F}$ (resp. $\mathcal{R}$) are closed under the aforementioned operations. With these definitions of addition and multiplication $(\setF, +, \cdot)$ and $(\setRscript, +, \cdot)$ are fields, and we can isomorphically embed the real numbers into both of them as a subfield using the map $\Pi:\setR \rightarrow \setF$ (resp. $\Pi:\setR \rightarrow \setRscript$) defined by
 \[
 	\Pi(r)[q] := \begin{cases}r &\mbox{ if } q=0\\0 &\mbox{ if } q\neq 0.\end{cases}
 \]

 \begin{dfn}(order on $\setF$ and $\setRscript$)
 	Let $x\ne y$ in $\setF$ (resp. in $ \setRscript$) be given. Then we say that $x$ is greater than $y$ and write $x>y$ if $(x-y)[\lambda(x-y)]>0$. We write $x<y$ if $y>x$ and $x\geq y$ if either $x=y$ or $x>y$.
 \end{dfn}
The relation $\ge$ defines a total order on $\setF$ and $\mathcal{R}$, which makes $(\setF, \geq)$ and $(\setRscript, \geq)$ into ordered fields. Moreover, the embedding of $\setR$ into these fields via the map $\Pi$ defined above is order preserving.

The map $|\cdot|_u: \mathcal{F} \rightarrow\setR$ (resp. $\mathcal{R} \rightarrow\setR$), given by
\[
|x|_u=\left\{\begin{array}{ll}e^{-\lambda(x)}&\mbox{if }x\ne 0\\
0&\mbox{if }x=0,\end{array}\right.
\]
is an ultrametric (non-Archimedean) valuation which induces on $\mathcal{F}$ (resp. on $\mathcal{R}$) the same topology as the order topology \cite{rsrevitaly13,barria-sham-18}; we will denote this topology by $\tau_v$ in the rest of the paper. The fields $\mathcal{F}$ and $\mathcal{R}$ are complete with respect to $\tau_v$ \cite{barria-sham-18}.

\begin{dfn}[$\ll$ and $\gg$]
	Let $x,y\in\setF$ (resp. $x,y\in \setRscript$) be nonnegative. Then we say that $x$ is infinitely larger than $y$ and write $x\gg y$ if for every $n\in\setN$, $x-ny > 0$; and we say that $x$ is infinitely smaller than $y$ and write $x\ll y$ if $y\gg x$. We say that $x$ is infinitely large if $x\gg 1$ and we say it is infinitely small or infinitesimal if $x\ll 1$.
\end{dfn}
Note that, in the above definition, $x\gg y$ if and only if $\lambda(x) < \lambda(y)$; in particular, $x$ is infinitely large if and only if $\lambda(x)<0$ and $x$ is infinitesimal if and only if $\lambda(x)>0$. The non-zero real numbers satisfy $\lambda(x)=0$ as does the sum of a real number and an infinitesimal number.

\begin{dfn}[The number d]
	We define the element $d\in\mathcal{R}\subset\setF$ as follows: for every $q\in\setQ$,
	\[
		d[q] := \begin{cases}1 &\mbox{ if } q=1\\ 0 &\mbox{ if } q\neq 1.\end{cases}
	\]
\end{dfn}
It follows from the above definition that $0<d\ll1$ (since $\lambda(d) = 1$);  and hence $d^{-1}\gg 1$  (in fact, $\lambda(d^{-1}) = -1$).

Since the Hahn field and Levi-Civita field are complete with respect to the valuation (order) topology $\tau_v$, it follows that infinite sums converge if and only if their terms form a null sequence. On one hand, this is convenient as it provides a simple convergence criterion for infinite series; but on the other hand this means that infinite series with real terms will converge if and only if they terminate, and power series with real coefficients (like the exponential, sine and cosine series) converge only in the infinitesimal neighbourhood around their centre.

On the Levi-Civita field this difficulty is overcome by defining a family of semi-norms and using them to induce a vector topology on the field that is weaker than $\tau_v$ and that turns $\mathcal{R}$ into a topological vector space \cite{rstopology09}.

\begin{dfn}[A Family of Semi-Norms on $\setRscript$]\label{lcsndef}
	For every $q\in\setQ$ define the map $\wkval{\cdot}{q}|_{\setRscript}:\setRscript\rightarrow\setR$ by
	\[
		\wkval{x}{q}:= \sup\left\{\abs{x[r]}\middle|r\in \supp(x)\cap(-\infty, q]\right\}.
	\]	
\end{dfn}

Since every $x\in\setRscript$ has left-finite support the supremum in the above definition is a maximum, it being the supremum of a finite set. This is not the case for $x\in\setF$ where the support of an element need only be well-ordered, which allows the support to have accumulation points and so for some $x\in\setF$, $\left\{\abs{x[r]}\middle|r\in \supp{x}\cap(-\infty, q]\right\}$ may have a divergent subsequence and thus the supremum may be $+\infty$. In the following sections we propose two different ways to overcome this difficulty and we show how they may be employed to induce a variety of topologies that are weaker than the order topology. Then we show how these topologies are related to each other, and finally we show what conditions must be satisfied to ensure that the induced topology has the same convergence criterion as the weak topology on the Levi-Civita field \cite{rstopology09,sham}.
	\section{Semi-norms}
One straightforward way to overcome the issue described above is simply to allow the semi-norm to take be equal to $+\infty$ in addition to values in $\setR$.

\begin{dfn}[Semi-Norms on $\setF$]\label{esndef}
	For every $r\in\setQ$ define a map $\unval{\cdot}{r}:\setF\rightarrow\setRext$ by
	\[
		\unval{x}{r}:=\sup\{\abs{x[q]}|q\leq r\}
	\]
\end{dfn}

The family of semi-norms from definition \ref{esndef} has the advantage that it reduces to \ref{lcsndef} when restricted to the Levi-Civita field. However, as we shall see, the topology induced on the Hahn field by this family is somewhat stronger than the weak topology on the Levi-Civita field. The following definitions will allow us to construct a similar family of semi-norms which, as we will see, has more useful properties.

\begin{dfn}[Well-Bounded Sets]
	Let $(S, \leq)$ be a totally ordered set such that every non empty subset $A\subset S$ has a maximum element. Then we say that $S$ is \textbf{well-bounded}.
\end{dfn}

\begin{dfn}[Well-Bounded Partition of \setQ]
	Let $\Gamma = \{\gamma_1, \ldots, \gamma_i, \ldots\}$ be a countable collection of mutually disjoint well-bounded subsets of $\setQ$ such that
	\[
		\bigcup\limits_{i=1}^{\infty}\gamma_i = \setQ.
	\]
	Then we say that $\Gamma$ is a \textbf{well-bounded partition} of $\setQ$. If in addition to the above we have that for every $i\in\setN$, $\gamma_i$ is finite then we call $\Gam$ a \textbf{finite well-bounded partition} of $\setQ$. For convenience we will use the notation
	\[
		\Gamma_n := \bigcup\limits_{i=1}^{n}\gamma_i.
	\]
	
\end{dfn}

\begin{xpl}[A well-bounded partition of \setQ]\label{diag_part}
	For every $n\in\setN$ define
	\[
		\gamma_n:= \left\{\frac{x}{y}\in\setQ\middle|x \mbox{ and } y \mbox{ are relatively prime}, \abs{x}+\abs{y}=n\right\},
	\]
	then $\Gamma = \{\gamma_i\}_{i\in\setN}$ is a well-bounded partition of \setQ. Clearly we have
	\[
		\bigcup\limits_{i=1}^{\infty}\gamma_i = \setQ
	\]
	since for any $q\in\setQ$ it is possible to find unique $x,y\in\setZ$ such that $x$ and $y$ are relatively prime and $\frac{x}{y}=q$. We then have by definition that $q\in\gamma_{\abs{x}+\abs{y}}$. Moreover, since every rational number has a unique reduced form, it follows that for any $i\neq j$ in $\setN$ , we have that $\gamma_i\cap\gamma_j = \emptyset$.
Finally, each $\gamma_i$ has only finitely many elements since
	\[
	\gamma_i \subseteq \left\{\frac{i-1}{1}, \frac{i-2}{2}, \ldots, \frac{1}{i-1}, -\frac{1}{i-1}, \ldots, -\frac{i-1}{1}\right\}.
	\]
Thus, the $\gamma_i$'s must be well-bounded. Therefore there is at least one well-bounded partition of \setQ.
\end{xpl}

\begin{table}[h]
\caption{The rational numbers}
\begin{center}
\begin{tabular}{ccccccccccccc}
\ldots&$\frac{-5}{1}$&$\mathbf{\frac{-4}{1}}$&$\frac{-3}{1}$                &$\frac{-2}{1}$                &$\frac{-1}{1}$                &$\frac{0}{1}$& $\frac{1}{1}$                &$\frac{2}{1}$                &$\frac{3}{1}$                &$\mathbf{\frac{4}{1}}$&$\frac{5}{1}$&\ldots\\
\ldots&\ldots              &                                        &$\mathbf{\frac{-3}{2}}$&                                        &$\frac{-1}{2}$                &                       & $\frac{1}{2}$                &                                      &$\mathbf{\frac{3}{2}}$&                                      &\ldots              &\ldots\\
\ldots&\ldots              &\ldots                               &                                        &$\mathbf{\frac{-2}{3}}$&$\frac{-1}{3}$                &                       & $\frac{1}{3}$                &$\mathbf{\frac{2}{3}}$&                                      &\ldots                             &\ldots              &\ldots\\
\ldots&\ldots              &\ldots                               &\ldots                               &\ldots                               &$\mathbf{\frac{-1}{4}}$&                       & $\mathbf{\frac{1}{4}}$&\ldots                             &\ldots                             &\ldots                             &\ldots              &\ldots\\
\ldots&\ldots              &\ldots                               &\ldots                               &\ldots                               &$\frac{-1}{5}$                &                       & $\frac{1}{5}$                &\ldots                             &\ldots                             &\ldots                             & \ldots             &\ldots\\
\end{tabular}
\end{center}
\caption*{The boldfaced numbers are the elements of $\gamma_5$.}
\end{table}

\begin{dfn}[The semi-norms on $\setF$ induced by a well-bounded partition of $\setQ$]\label{psndef}
	Let $\Gam$ be any well-bounded partition of $\setQ$ and for every $n\in\setN$ define the map $\gamval{\cdot}{n}:\setF\rightarrow\setR$ by
	\[
		\gamval{x}{n}:= \max\left\{\abs{x[q]}\middle|q\in\Gamn{n}\right\}.
	\]
\end{dfn}

Note that, because the support of $x$ is well-ordered and \Gamn{n} is well-bounded the set $\left\{\val{x[q]}\middle|q\in\Gamn{n}\right\}$ in the defnition above contains only finitely many non-zero elements and hence the maximum does exist. We have not yet shown that either definition \ref{esndef} or \ref{psndef} actually define semi-norms, to avoid unnecessary repetition we include only the proof of this for the latter definition; the proof for the former follows similarly.

\begin{prp}
	Let $\Gam$ be any well-bounded partition of $\setQ$ and let $n\in\setN$. Then, $\gamval{\cdot}{n}$ is a semi-norm.
\end{prp}

\begin{proof}
We need to show that, for every $x,y\in\setF$ and for every $a\in\setR$, we have that
\begin{itemize}
	\item $\gamval{x}{n}\geq 0$
	\item $\gamval{ax}{n} = \abs{a}\gamval{x}{n}$
	\item $\gamval{x+y}{n}\leq \gamval{x}{n}+\gamval{y}{n}$.
\end{itemize}
The first property follows trivially from the definition. Now let $a\in\setR$ and $x\in\setF$ be given. Then
\begin{align*}
	\gamval{ax}{n} &= \max\left\{\abs{ax[q]}\middle|q\in\Gamn{n}\right\}\\
	&= \max\left\{\abs{a}\abs{x[q]}\middle|q\in\Gamn{n}\right\}\\
	&= \abs{a}\max\left\{\abs{x[q]}\middle|q\in\Gamn{n}\right\}\\
	&=\abs{a}\gamval{x}{n},
\end{align*}

Finally, let $x,y\in\setF$ be given. Then
\begin{align*}
	\gamval{x+y}{n} &= \max\left\{\abs{(x+y)[q]}\middle|q\in\Gamn{n}\right\}\\
	&\leq\max\left\{\abs{x[q]}+\abs{y[q]}\middle|q\in\Gamn{n}\right\}\\
	&\leq \max\left\{\abs{x[q]}\middle|q\in\Gamn{n}\right\} + \max\left\{\abs{y[q]}\middle|q\in\Gamn{n}\right\}\\
	&=\gamval{x}{n}+\gamval{y}{n}.
\end{align*}
\end{proof}
	\section{Vector topologies}
Having defined the semi-norms we will be working with, we proceed to show that both families can be used to induce vector topologies on $\mathcal{F}$ that are consistent with a translation invariant metric. Naturally the proofs are very similar in both cases so we will present them each once making notes where there are significant differences or modifications that must be accounted for. Note that, unless otherwise stated, $\Gam$ will denote any arbitrary well-bounded partition of $\setQ$ and  $\gamval{\cdot}{n\in\setN}:\setF\rightarrow\setR$ will denote the corresponding family of semi-norms.

\begin{rmk}
	Note that, if $r_1 < r_2$ then $\{\abs{x[q]}|q\leq r_1\}\subset\{\abs{x[q]}|q\leq r_2\}$ for all $x\in\setF$. It follows immediately that
	\[
		\unval{x}{r_1}\leq\unval{x}{r_2}.
	\]
	Similarly if $n,m\in\setN$ with $n<m$, then $\Gamn{n}\subset\Gamn{m}$ and hence for any $x\in\setF$
	\[
		\gamval{x}{n}\leq\gamval{x}{m}.
	\]
\end{rmk}

\begin{dfn}[Pseudo-Ball]
	Let $x\in\setF$, and let $r>0$ in $\setR$ be given (resp. let $q>0$ in $\setQ$ be given). Then we define
	\[
		\PB_{\Gam}\left(x,r\right):= \left\{y\in\setF\middle|\gamval{x-y}{\mu(r)}<r\right\}
	\]
	where
	\[
		\mu(r) := \left\lceil\frac{1}{r}\right\rceil
	\]
	is the smallest natural number $n$ such that $\frac{1}{n}<r$. We say that $\PB_{\Gam}(x,r)$ is the \textbf{``pseudo-ball''} at $x$ with radius $r$.
	Respectively we define
	\[
		\pbu{x}{q}:=\{y\in\setF|\unval{y-x}{1/q}<q\}
	\]
	and we call this a \textbf{``pseudo-ball''} at $x$ with radius $q$.
\end{dfn}

\begin{prp}
	Let $x\in\setF$ and let $0<r_1<r_2\in\setR$ be given (resp. let $0<r_1<r_2\in\setQ$ be given). Define $r:= \min\{r_1, r_2-r_1\}$, then for all $y\in\pbu{x}{r}$ we have that
	\[
		\pbg{y}{r_1} \subset \pbg{x}{r_2};
	\]
	in particular we have that
	\[
		\pbg{x}{r_1}\subset\pbg{x}{r_2}.
	\]
	Respectively, we have that
	\[
		\pbu{y}{r_1}\subset\pbu{x}{r_2};
	\]
	and hence
	\[
		\pbu{x}{r_1}\subset \pbu{x}{r_2}.
	\]
\end{prp}

\begin{proof}
	Let $y\in\pbg{x}{r}$ be given and let $z\in\pbg{y}{r_1}$. Then, by definition, we have that
	\[
		\gamval{y-z}{\mu(r_1)}<r_1.
	\]
	Since $r_1<r_2$, it follows that $\mu(r_1)\geq\mu(r_2)$, and hence
	\[
		\gamval{x-z}{\mu(r_2)}\leq\gamval{x-z}{\mu(r_1)}.
	\]
	It follows that
	\begin{align*}
		\gamval{x-z}{\mu(r_1)} &\leq \gamval{y-z}{\mu(r_1)} + \gamval{x-y}{\mu(r_1)}\\
		&< r_1 + \gamval{x-y}{\mu(r_1)}.
	\end{align*}
	Recall that $r = \min\{r_1, r_2-r_1\}\leq r_1$, and hence
	\[
		\gamval{x-y}{\mu(r_1)} \leq \gamval{x-y}{\mu(r)}.
	\]
	Since $y\in\pbg{x}{r}$, we have that
	\[
		\gamval{x-y}{\mu(r)} < r \leq r_2-r_1.
	\]
	Altogether, it follows that
	\begin{align*}
		\gamval{x-z}{\mu(r_2)} &<  r_1 + \gamval{x-y}{\mu(r_1)}\\
		&\leq r_1 +\gamval{x-y}{\mu(r)}\\
		&< r_1 + (r_2 - r_1)\\
		&=r_2;
	\end{align*}
and hence
$z\in\pbg{x}{r_2}$.
	This argument holds for any $z\in\pbg{y}{r_1}$, and hence
	\[
		\pbg{y}{r_1} \subset \pbg{x}{r_2}.
	\]
In particular, letting $y=x$ in $\pbg{x}{r}$, we have that
	\[
		\pbg{x}{r_1} \subset \pbg{x}{r_2}.
	\]
\end{proof}

We can now define the topologies induced by these families of semi-norms by letting a set $S$ be open if every point in $S$ is contained in a pseudo-ball which is itself contained in $S$.

\begin{dfn}[The topologies induced by families of semi-norms]
	We define
	\[
		\topg:= \left\{O\subset\setF\middle| \forall x\in O, \exists r>0 \mbox{ in }\setR \mbox{ such that } \pbg{x}{r}\subset O\right\},
	\]
	and
	\[
		\topu := \left\{O\subset\setF\middle|\forall x\in O, \exists r>0 \mbox{ in }\setR \mbox{ such that } \pbu{x}{r}\subset O\right\}.
	\]
	We call these the \textbf{topology induced by $\gam$} and the \textbf{locally uniform support topology} respectively. The name of the latter topology will be justified when we discuss the convergence criterion in this topology.
\end{dfn}

\begin{prp}
	$\topg$ is a topology on \setF (resp. $\topu$ is a topology on \setF).
\end{prp}

\begin{proof}
	We need to show that $\topg$ is closed under arbitrary unions and finite intersections, and that $\emptyset,\setF\in\topg$. Let $\left\{O_\alpha\right\}_{\alpha\in A}$ be an arbitrary collection of elements of $\topg$; and let $x\in\bigcup\limits_{\alpha\in A}O_{\alpha}$ be given. Then there is an $\alpha_0\in A$ such that $x\in O_{\alpha_0}$. But $O_{\alpha_0}\in\topg$ so by definition there is a $r>0$ in $\setR$ such that $\pbg{x}{r}\subset O_{\alpha_0}$. It follows immediately that $\pbg{x}{r}\subset\bigcup\limits_{\alpha\in A}O_{\alpha}$. Thus, $\bigcup\limits_{\alpha\in A}O_{\alpha}$ is open, and hence $\topg$ is closed under arbitrary unions.

	 Now let $O_1,O_2\in\topg$ and let $x\in O_1\cap O_2$ be given. Since $x\in O_1$, there exists $r_1>0$ in $\setR$ such that $\pbg{x}{r_1}\subset O_1$. Similarly there exists $r_2>0$ in $\setR$ such that $\pbg{x}{r_2}\subset O_2$. Let $r=\min\left\{r_1, r_2\right\}$. Then we have that $\pbg{x}{r}\subset \pbg{x}{r_1}\subset O_1$ and $\pbg{x}{r}\subset \pbg{x}{r_2}\subset O_2$; and hence $\pbg{x}{r}\subset O_1 \cap O_2$. This shows that $\topg$ is closed under the intersection of two of its elements; and by induction it is therefore closed under finite intersections.

Finally, that $\emptyset$ and $\setF$ are in  $\topg$ follows from that the fact that they both trivially satisfy the defining property of $\topg$.
\end{proof}

\begin{prp}
	$\left(\setF, \topg\right)$ (resp. $\left(\setF, \topu\right)$) is a topological vector space.
\end{prp}

\begin{proof}
	To show that $\left(\setF, \topg\right)$ is a topological vector space, we will prove the following.
	\begin{itemize}
		\item Every singlton is closed with respect to \topg.
		\item Vector addition is continuous with respect to \topg.
		\item Scalar multiplication is continuous with respect to \topg.
	\end{itemize}
	
Let $x\in\setF$ be given. We will show that $\{x\}$ is closed in $\topg$ by showing that its complement is open. So, let $y\in\setF\setminus\{x\}$ be given, let $q = \lambda(x-y)$, and let $r_0:= \abs{(x-y)[q]}$. Choose $N\in\setN$ large enough so that $q\in\Gamn{N}$, then let $r = \min\left\{\frac{1}{N}, \frac{r_0}{2}\right\}$. We will show that $\pbg{y}{r}\subset\setF\setminus\{x\}$ by showin that $x\notin\pbg{y}{r}$. Note that
	\begin{align*}
		\gamval{x-y}{\mu(r)} &\geq \gamval{x-y}{N}\\
		&=\max\left\{\abs{(x-y)[q]}\middle|q\in\Gamn{N}\right\}\\
		&\geq r_0\\
		&> r.
	\end{align*}
	Thus, $x\notin\pbg{y}{r}$, and hence$\setF\setminus\{x\}$ is open.
	
	Next, we show that $+:\setF\times\setF\rightarrow\setF$ is a continuous operation on $(\setF, \topg)\times (\setF, \topg)$. Let $O\subset\setF$ be any open set with respect to \topg, let $A\subset\setF\times\setF$ be the inverse image of $O$ under addition. We will show that $A$ is open in $(\setF, \topg)\times (\setF, \topg)$. Fix $(x_1, x_2)\in A$, then $x_1+x_2\in O$. $O$ is open so there exists a $r>0$ in $\setR$ such that $\pbg{x_1+x_2}{r}\subset O$. Let $y\in\pbg{x_1}{\frac{r}{2}}$, $z\in\pbg{x_2}{\frac{r}{2}}$. Then,
	\begin{align*}
		\gamval{y+z-x_1-x_2}{\mu(r)} &\leq \gamval{y-x_1}{\mu(r)} + \gamval{z-x_2}{\mu(r)}\\
		&\leq  \gamval{y-x_1}{\mu\left(\frac{r}{2}\right)} + \gamval{z-x_2}{\mu\left(\frac{r}{2}\right)}\\
		&< \frac{r}{2} + \frac{r}{2} = r.
	\end{align*}
	Thus, $y+z\in O$ and hence $(y,z)\in A$. Therefore, $A\subset\setF\times\setF$  is open with respect to $(\setF, \topg)\times (\setF, \topg)$, and addition is continuous.
	
Finally we show that $\cdot:\setR\times\setF\rightarrow\setF$ is continuous with respect to \topg. So, let $O\subset\setF$ be open with respect to $\topg$ and let $S\subset\setR\times\setF$ be the inverse image of $O$ under scalar multiplication. Let $(\alpha, x)\in S$, then $\alpha x \in O$ and since $O$ is open there is a $r>0$ in $\setR$ such that $\pbg{\alpha x}{r}\subset O$. Now we have two cases; either $\alpha=0$ or $\alpha\neq0$. We deal with these cases separately.
	
	Assume that $\alpha = 0$. Then we have two possibilities: either $\gamval{x}{\mu(r)} = 0$ or $\gamval{x}{\mu(r)}\neq0$. Consider first the subcase  $\gamval{x}{\mu(r)} = 0$; then we will show that $(-1, 1)\times\pbg{x}{r} \subset S$. Let $\beta\in(-1,1)$ and $y\in\pbg{x}{r}$ be given, then
	\begin{align*}
		\gamval{\beta y}{\mu(r)} &= \abs{\beta}\gamval{y}{\mu(r)}\\
		&<\gamval{y}{\mu(r)}\\
		&\leq \gamval{y-x}{\mu(r)} + \gamval{x}{\mu(r)}\\
		&= \gamval{y-x}{\mu(r)} < r.
	\end{align*}
	Thus, $\beta y\in\pbg{0}{r}\subset O$ and hence $(\beta, y)\in S$. Next we consider $\gamval{x}{\mu(r)}\neq0$; let
	\[
		r_1=\min\left\{\frac{1}{2}, \frac{r}{2\gamval{x}{\mu(r)}}\right\}.
	\]
	Then $r_1\in\setR$ and $r_1>0$; and we will show that $(-r_1, r_1)\times\pbg{x}{r}\subset S$. So let $\beta\in(-r_1, r_1)$ and $y\in\pbg{x}{r}$ be given. Then
	\begin{align*}
		\gamval{\beta y}{\mu(r)} &\leq \gamval{\beta(y-x)}{\mu(r)} +\gamval{\beta x}{\mu(r)}\\
		&\leq \abs{\beta}\gamval{y-x}{\mu(r)} + \abs{\beta}\gamval{x}{\mu(r)}\\
		&<r_1r + r_1\gamval{x}{\mu(r)}\\
		&\leq \frac{r}{2} + \frac{r}{2\gamval{x}{\mu(r)}}\gamval{x}{\mu(r)}\\
		&= r.
	\end{align*}
	This, $\beta y\in O$ and hence $(\beta, y)\in S$.
	
Now we consider the case $\alpha\neq0$. Let
	\[
		r_1=\min\left\{\frac{r}{2}, \frac{r}{2\abs{\alpha}}\right\}
	\]
	and
	\[
		\nu:= \begin{cases} \frac{1}{2} &\mbox{ if } \gamval{x}{\mu(r)}=0\\\min\left\{\frac{1}{2}, \frac{r}{4\gamval{x}{\mu(r)}}\right\} &\mbox{ if } \gamval{x}{\mu(r)}\neq0\end{cases}.
	\]
	We will show that $(\alpha-\nu, \alpha+\nu)\times\pbg{x}{r_1}\subset S$. So let $\beta\in(\alpha-\nu, \alpha+\nu)$ and $y\in\pbg{x}{r_1}$ be given. Then
	\begin{align*}
		\gamval{\beta y - \alpha x}{\mu(r)} &= \gamval{(\beta-\alpha)(y-x)+(\beta-\alpha)x +\alpha(y-x)}{\mu(r)}\\
		&\leq \abs{\beta-\alpha}\gamval{y-x}{\mu(r)} + \abs{\beta-\alpha}\gamval{x}{\mu(r)} + \abs{\alpha}\gamval{y-x}{\mu(r)}.
	\end{align*}
	However, $r_1\leq\frac{r}{2}<r$ and hence
	\[
		\gamval{y-x}{\mu(r)} \leq \gamval{y-x}{\mu(r_1)}<r_1\leq\frac{r}{2\abs{\alpha}}.
	\]
	Thus,
	\[
		\abs{\alpha}\gamval{y-x}{\mu(r)} < \frac{r}{2}.
	\]
	Moreover,
	\[
		\abs{\beta-\alpha}\gamval{y-x}{\mu(r)} < \abs{\beta-\alpha}r_1 < \nu r_1 \leq \frac{r}{4}.
	\]
	And finally,
	\[
		\abs{\beta-\alpha}\gamval{x}{\mu(r)} < \nu\gamval{x}{\mu(r)} \leq \frac{r}{4}.
	\]
	Thus, altogether, we obtain that
	\begin{align*}
		\gamval{\beta y - \alpha x}{\mu(r)}&\leq \abs{\beta-\alpha}\gamval{y-x}{\mu(r)} + \abs{\beta-\alpha}\gamval{x}{\mu(r)} + \abs{\alpha}\gamval{y-x}{\mu(r)}\\
		&< \frac{r}{4} + \frac{r}{4} + \frac{r}{2} =r.
	\end{align*}
	So $\beta y\in O$ and hence $(\beta, y)\in S$. Therefore we conclude that $(\setF, \topg)$ is a topological vector space.
\end{proof}

\begin{prp}\label{clbut}
	The family of pseudo-balls $\left\{\pbu{0}{q}\middle|q\in\setQ^{+}\right\}$ (resp. the family of pseudo-balls $\left\{\pbg{0}{q}\middle|q\in\setQ^{+}\right\}$) is a countable local base for $\topu$ (resp. $\topg$) at $0$.
\end{prp}

\begin{proof}
	Let $O\in\topu$ be any open set in $\setF$ containing $0$. Then there exists $r>0$ in $\setR$ such that
	$\pbu{0}{r}\subset O$.
	Let $q\in\setQ$ be such that $0<q<r$; then $\pbu{0}{q}\subset\pbg{0}{r}\subset O$. Thus, for any open set containing $0$, there is a $q\in\setQ$ such that $0\in\pbu{0}{q}\subset O$ and hence $\left\{\pbu{0}{q}\middle|q\in\setQ^{+}\right\}$ is a countable local base for $\topu$ at $0$.
\end{proof}

\begin{crl}
	For any $x\in\setF$, the family of pseudo-balls $\left\{\pbu{x}{q}\middle|q\in\setQ^{+}\right\}$ (resp. $\left\{\pbg{x}{q}\middle|q\in\setQ^{+}\right\}$)  is a countable local base for $\topu$ (resp. $\topg$) at $x$.
\end{crl}
	\section{Relations between topologies}
	Now that we have established that $\topg$ and $\topu$ are vector topologies, we will investigate their relationship to each other and to $\topv$. We begin by recalling the definition of a compact set.
\begin{dfn}
	Suppose $\tau$ is a topology on $\setF$ and let $A\subset\setF$. Then we say that $A$ is compact in $(\setF, \tau)$ if every open cover of $A$ in $(\setF, \tau)$ admits a finite subcover.
\end{dfn}

\begin{prp}
	Let $\tau$ be any topology on $\setF$ satisfying $\tau\subsetneq\topv$. Suppose $A\subset\setF$ is compact in $(\setF, \topv)$, then $A$ is also compact in $(\setF, \tau)$.
\end{prp}

\begin{proof}
	Let $A\subset\setF$ a compact set in $(\setF, \topv)$ be given. Let $T\subset\tau$ be any open cover of $A$ in $\tau$, then since $\tau\subsetneq\topv$ and $T\subset\tau$ we have that $T\subset\topv$. Thus $T$ is an open cover of $A$ in $\topv$, however by choice $A$ is compact in $(\setF, \topv)$ so $T$ must admit a finite subcover $T'\subset T$. But $T'\subset T\subset\tau$ so $T'$ is also a finite subcover in $(\setF, \tau)$. This argument holds for any choice of $A\subset\setF$ compact in $(\setF, \topv)$; thus, if $A$ is compact in $(\setF, \topv)$ it is also compact in $(\setF, \tau)$.
\end{proof}

\begin{prp}
	$\topg \subsetneq \topv$ (resp. $\topu \subsetneq \topv$).
\end{prp}

\begin{proof}
	Let $G\subset\setF$ be open with respect to $\topg$ and fix $x\in G$. Then there exists $r>0$ in $\setR$ such that
$\pbg{x}{r}\subset G$.	Let $n>\max\{\Gam_{\mu(r)}\}$ which is possible because $\Gam_{\mu(r)}$ is the finite union of well-bounded sets and hence it is itself well-bounded. We will show that $B(x,d^n) \subset G$.
So let $y\in B(x, d^n)$ be given. Then clearly we have
$\abs{y-x}<d^n$ and hence, for any $q<n$ in $\setQ$, we have $(x-y)[q] = 0$. However, by our choice of $n$, we have that for every $q\in\Gamn{\mu(r)}$, $q<n$. Therefore, for every $q\in\Gamn{\mu(r)}$, $(y-x)[q]=0<r$. It follows that $y\in\pbg{x}{r}$. This holds for any $y\in B(x, d^n)$. It follows that
	\[
		B(x,d^n)\subset\pbg{x}{r}\subset G.
	\]
	
We have just shown that $\topg\subset\topv$; so it remains to show that there exists an $O\in\topv$ such that $O\notin\topg$. Choose $n>\max\{\Gamn{1}\}=q$, then
	\[
		\left(-d^n, d^n\right) = B(0, d^n)\in\topv.
	\]
	Now fix $r>0$ in $\setR$ and let $x= \frac{r}{2}d^q$. Then clearly $x\notin(-d^n, d^n)$ since by choice $d^q\gg d^n$. However
	\[
		\gamval{0-x}{\mu(r)} = \gamval{ \frac{r}{2}d^q}{\mu(r)} \leq \frac{r}{2} < r
	\]
	so $x\in\pbg{0}{r}$. Since our choice of $r$ was arbitrary we conclude that for every $r>0$ in \setR, $ \frac{r}{2}d^q\in\pbg{0}{r}$ but $ \frac{r}{2}d^q\notin \left(-d^n, d^n\right)$. It follows that for every $r>0$ in \setR,
	\[
		\pbg{0}{r}\not\subset \left(-d^n, d^n\right).
	\]
Thus, $(-d^n, d^n) \notin \topg$ and hence $\topv\not\subset\topg$.
\end{proof}

\begin{crl}
	Suppose $A\subset\setF$ is compact in $(\setF, \topv)$. Then $A$ is also compact in $(\setF, \topu)$ and $(\setF, \topg)$.
\end{crl}

\begin{prp}
	There exist translation invariant metrics $\Dg$ and $\Dw$ that induce topologies which are equivalent to $\topg$ and $\topu$, respectively, on $\setF$.
\end{prp}

\begin{proof}
	This follows from the fact that both $(\setF, \topg)$ and $(\setF, \topu)$ are topological vector spaces with countable local bases, for details see Theorem 1.24 in \cite{rudin1}.
\end{proof}

\begin{xpl}
	Let
	\[
		\dg{x}{y} := \sum\limits_{k=1}^{\infty}2^{-k}\frac{\gamval{x-y}{k}}{1+\gamval{x-y}{k}}
	\]
	and let
	\[
		\du{x}{y} := \sum\limits_{k=1}^{\infty}2^{-k}\frac{\unval{y-x}{k}}{1+\unval{y-x}{k}}.
	\]
	Then $\Dg$ and $\Du$ are translation invariant metrics on $\setF$ that induce topologies equivalent to $\topg$ and $\topu$, respectively; see the proofs of Theorem 3.32 and Theorem 3.33 in \cite{rstopology09}. We note that, in both infinite sums above, the $2^{-k}$ factor could be replace with $c^{-k}$ where $c$ is any real number greater than 1 and we would still obtain translation invariant metrics.
\end{xpl}

\begin{prp}\label{prpref1}
	$\topg\subset\topu$.
\end{prp}

\begin{proof}
	Let $O\in\topg$ be given and fix $x\in O$. Then there exists $r>0$ in $\setR$ such that $\pbg{x}{r}\subset O$. Let $q_0 = \max\{\Gamn{\mu(r)}\}$, which exists by the well-boundedness of $\Gamn{\mu(r)}$. Pick $q_1>\max\{q_0, \frac{1}{r}\}$ we claim that
	\[
		\pbu{x}{\frac{1}{q_1}}\subset \pbg{x}{r}.
	\]
	So fix $y\in\pbu{x}{\frac{1}{q_1}}$, then $\unval{y-x}{q_1}< \frac{1}{q_1} < r$. Thus for every $q\leq q_1$, we have that $(y-x)[q]<r$. However, by selection, we have that $q_1>q_0= \max\{\Gamn{\mu(r)}\}$; thus, for every $q\in\Gamn{\mu(r)}$ we have that $q\leq q_0< q_1$ and hence, for every $q\in\Gamn{\mu(r)}$, we have that $(y-x)[q]<r$. It that $\gamval{y-x}{\mu(r)}<r$ and hence $y\in\pbg{x}{r}$. Therefore, as claimed, $\pbu{x}{\frac{1}{q_1}}\subset \pbg{x}{r}\subset O$. Thus, $O\in\topu$ and hence $\topg\subseteq\topu$.
\end{proof}

\begin{prp}
	Let $\Gam$ and $\Omega$ be a finite well-bounded partition and an infinite well-bounded partition of $\setQ$, respectively. Then $\topg \subsetneq \Top{\Omega}$.
\end{prp}

\begin{proof}
	First we show that $\topg\subset\Top{\Omega}$. So let $O\in\topg$ and fix $x\in O$. Then there exists $\epsilon>0$ in $\setR$ such that $\pbg{x}{\epsilon}\subseteq O$. Since $\Gam$ is a finite partition we know that $\Gamn{\mu(\epsilon)}$ is a finite set and thus there exists $N_0\in\setN$ such that $\Gamn{\mu(\epsilon)} \subset \Omega_{N_0}$. Let $N_1\in\setN$ be large enough so that $\frac{1}{N_1}<\epsilon$; and let $N=\max\left\{N_0, N_1\right\}$. We claim that $\pb{\Omega}{x}{\frac{1}{N}}\subset \pbg{x}{\epsilon}$. Let  $y\in\pb{\Omega}{x}{\frac{1}{N}}$ be given. By definition, we have that
	\[
		\genval{y-x}{\Omega}{N} < \frac{1}{N} \leq \frac{1}{N_1} < \epsilon.	
	\]
	So for every $q\in\Omega_{N}$, we have that
$|(y-x)[q]| < \epsilon$.
Since $\Gamn{\mu(\epsilon)} \subset \Omega_{N_0}\subset \Omega_{N}$, it follows that, for every $q\in\Gamn{\mu(\epsilon)}$, we have that
	\[
		|(y-x)[q]| < \epsilon.
	\]
Thus,
	\[
		\gamval{y-x}{\mu(\epsilon)} < \epsilon
	\]
and hence $y\in\pbg{x}{\epsilon}$. Hence $\pb{\Omega}{x}{\frac{1}{N}}\subset \pbg{x}{\epsilon}\subseteq O$. Thus, we have shown that, for every $x\in O$, there exists $N\in\setN$ such that
	\[
	\pb{\Omega}{x}{\frac{1}{N}} \subseteq O.
	\]
	This proves that $\topg \subset \Top{\Omega}$. To prove that the two topologies are not equal, we observe that since $\Omega$ is non-finite there exists $N\in\setN$ such that for all $n\geq N$, $\Omega_{n}$ has infinitely many elements. Consider the pseudo-ball $\pb{\Omega}{0}{\frac{1}{N}} \in \Top{\Omega}$. We will show that $\pb{\Omega}{0}{\frac{1}{N}}\not\in \topg$. Since we have already shown that the family of pseudo-balls $\left\{\pbg{0}{q}\mid|q\in\setQ^{+}\right\}$ is a countable local basis for $\topg$ at $0$, it is enough to prove that for every $q\in\setQ^{+}$, $\pbg{0}{q}\not\subset\pb{\Omega}{0}{\frac{1}{N}}$. So fix a $q\in\setQ^{+}$. Since $\Gamn{\mu(q)}$ is a finite set and $\Omega_{N}$ is not, there must be $q_0\in\Omega_{N}\setminus\Gamn{\mu(q)}$; let $x_0 = \frac{2}{N}d^{q_0}$. We have that
	\[
		\gamval{x_0}{\mu(q)} = \sup\left\{\abs{x_0[q']}\middle|q'\in\Gamn{\mu(q)}\right\} = 0 < q
	\]
	because $x_0[q']=0$ for all $q'\in\Gamn{\mu(q)}$. However, we also have that
	\[
		\genval{x_0}{\Omega}{N} = \sup\left\{\abs{x_0[q']}\middle|q'\in \Omega_{N}\right\} = \frac{2}{N} > \frac{1}{N}.
	\]
	Thus, for every $q\in\setQ^{+}$, there is $x_0\in\setF$ such that $x_0\in\pbg{0}{q}$ but $x_0 \not\in \pb{\Omega}{0}{\frac{1}{N}}$. It follows that $\Top{\Omega}\not\subset \topg$ and hence $\topg \subsetneq \Top{\Omega}$.
\end{proof}

\begin{lmm}\label{lemmaref1}
	For every $q\in\setQ^{+}$ and for every $n\in\setN$ there exists $q'\in(0, q]\cap\setQ$ such that $q'\not\in\Gamn{n}$.
\end{lmm}

\begin{proof}
	Suppose otherwise. Then there exist $q\in\setQ$ and $n\in\setN$ such that for every $q'\in(0, q]\cap\setQ$, $q'\in\Gamn{n}$. As an immediate consequence we have that $(0, q]\cap\setQ\subset\Gamn{n}$, which contradicts the fact that intervals in $\setQ$ are not well-bounded. Thus, no such $q$ and $n$ can exist.
\end{proof}

\begin{prp}\label{prpref2}
	$\topu\not\subset\topg$.
\end{prp}

\begin{proof}
	Fix $q\in\setQ^{+}$ and consider the pseudo-ball $\pbu{0}{\frac{1}{q}}\in\topu$. We claim that for every $r\in\setQ^{+}$, $\pbg{0}{r}\not\subset\pbu{0}{\frac{1}{q}}$. Since $\left\{\pbg{0}{r}\middle|r\in\setQ^{+}\right\}$ forms a local base for $\topg$ at $0$, proving our claim will be sufficient to establish that $\pbu{0}{\frac{1}{q}}\not\in\topg$. To prove our claim, let $r\in\setQ^{+}$. Then
	\begin{align*}
		\pbg{0}{r} &= \left\{x\in\setF\middle|\gamval{x}{\mu(r)}<r\right\}\\
		&= \left\{x\in\setF\mid|\sup\left\{\abs{x[q]}\middle|q\in\Gamn{\mu(r)}\right\}<r\right\}.
	\end{align*}
	 By lemma \ref{lemmaref1}, we have that $(-\infty, q]\cap(\setQ\setminus\Gamn{\mu(r)}) \neq \emptyset$; so pick $q'\in(-\infty, q]\cap(\setQ\setminus\Gamn{\mu(r)})$ and let $x=\frac{2}{q}d^{q'}$. We see that
	 \[
	 	x\in\pbg{0}{r}
	 \]
	 since $x[q] = 0$ for all $q\in\Gamn{\mu(r)}$, but
	 \[
	 	x\not\in\pbu{0}{\frac{1}{q}}
	 \]
	 because $q'\in(-\infty, q]\cap\setQ$ and $\abs{x[q']} = \frac{2}{q} > \frac{1}{q}$. Thus the claim (and hence the proposition) is proved.
\end{proof}

Combining the results of Proposition \ref{prpref1} and Proposition \ref{prpref2}, we readily obtain the following corollary.
\begin{crl}
	$\topg\subsetneq\topu$.
\end{crl}

\begin{prp}
	Let $\Gam$ and $\Omega$ be distinct finite well-bounded partitions of $\setQ$. Then $\topg = \Top{\Omega}$.
\end{prp}

\begin{proof}
	Let $O\in\Top{\Omega}$ and fix $x\in O$. Then there exists $\epsilon>0$ in $\setR$ such that
	\[
		\pb{\Omega}{x}{\epsilon}\subseteq O.
	\]
	$\Omega$ is a finite partition so $\Omega_{\mu(\epsilon)}$ contains only a finite number of elements. Moreover, since $\Gam$ is also a partition of $\setQ$, for each $q\in\Omega_{\mu(\epsilon)}$ there exists $N_q\in\setN$ such that $q\in\Gamn{N_q}$. Let
	\[
		N=\max\left\{N_q\middle|q\in\Omega_{\mu(\epsilon)}\right\}.
	\]
Then $\Omega_{\mu(\epsilon)}\subset\Gamn{N}$
	because for every $q\in\Omega_{\mu(\epsilon)}$, we have that $q\in\Gamn{N_q}\subseteq\Gamn{N}$.
	Now let $\delta = \min\left\{\frac{1}{N}, \epsilon\right\}$. We claim that
	\[
		\pbg{x}{\delta}\subseteq\pb{\Omega}{x}{\epsilon}.
	\]
	So let $y\in\pbg{x}{\delta}$ be given. Then we have that
$\gamval{y-x}{\mu(\delta)} < \delta$,
and hence, for every $q\in\Gamn{\mu(\delta)}$, we have that
$\abs{(y-x)[q]} < \delta$.
But $\delta\leq\frac{1}{N}$ so $\mu(\delta)\geq N$, and hence $\Gamn{N}\subseteq\Gamn{\mu(\delta)}$.

We have already shown that $\Omega_{\mu(\epsilon)}\subseteq\Gamn{N}\subseteq\Gamn{\mu(\delta)}$. So for every $q\in\Omega_{\mu(\epsilon)}$, we have that
$\abs{(y-x)[q]} < \delta \leq \epsilon$.
It follows that
	\[
		\genval{y-x}{\Omega}{\mu(\epsilon)} < \epsilon,
	\]
and hence $y\in\pb{\Omega}{x}{\epsilon}$. The above argument holds for any $y\in\pbg{x}{\delta}$; it follows that $\pbg{x}{\delta}\subset\pb{\Omega}{x}{\epsilon}\subset O$. Thus, we have shown that for any $x\in O$ there exists $\delta>0$ in $\setR$ such that $\pbg{x}{\delta}\subset O$, and hence $O\in\topg$. Since $O$ is an arbitrary element of $\Top{\Omega}$, we infer that $\Top{\Omega}\subseteq\topg$. A symmetric argument shows that $\topg\subseteq\Top{\Omega}$, and hence $\topg = \Top{\Omega}$.
\end{proof}

\begin{crl}
	All finite well-bounded partitions of $\setQ$ induce the same topology on $\setF$.
\end{crl}

\begin{dfn}
	We call the topology induced by finite well-bounded partitions of $\setQ$ the \textbf{weak topology} and denote it by $\topw$. In a later section we will justify this choice of name by showing that $\topw$ share the same convergence criterion as the weak topology on the Levi-Civita field. A detailed study of the weak topology on the Levi-Civita field $\mathcal{R}$ can be found in \cite{shamberz, rstopology09}.
\end{dfn}
	\section{Convergence of sequences}
In this section we will study convergence in $\topu$, $\topg$, and in particular in $\topw$. We start with the following definition.

\begin{dfn}
	Let $\tau$ be any topology on $\setF$ induced by a countable family of semi-norms (which we will denote by $\genval{\cdot}{\tau}{n}$) and let $\left(s_n\right)_{n\in\setN}$ be a sequence in $\setF$. We say that $\left(s_n\right)_{n\in\setN}$ converges in $\tau$ if there exists $s\in\setF$ such that for every $\epsilon>0$ in $\setR$ there exists $N\in\setN$ such that if $n\geq N$ then
$\genval{s_n-s}{\tau}{\mu(\epsilon)} < \epsilon$.
\end{dfn}

We would like to find necessary conditions for a sequence to converge as defined above but first we introduce the following useful notion of regularity.

\begin{dfn}[Regular sequence (Hahn field)]
	Let $\left(s_n\right)_{n\in\setN}$ be a sequence in $\setF$. If
	\[
		\bigcup\limits_{n=1}^{\infty}\supp\left(s_n\right)
	\]
	is well-ordered then we say that $\left(s_n\right)_{n\in\setN}$ is a \textbf{regular sequence}.
\end{dfn}

For comparison, regularity of a sequence in the Levi-Civita field is defined as follows:

\begin{dfn}[Regular sequence (Levi-Civita field)]
	Let $\left(x_n\right)_{n\in\setN}$ be a sequence in $\setRscript$. We say that $(x_n)$ is a \textbf{regular sequence} in the Levi-Civita field if
	\[
		\bigcup_{n=1}^{\infty}\supp\left(x_n\right)
	\]
	is left-finite.
\end{dfn}

We are now ready to state necessary conditions for a sequence in $\setF$ to converge in each of $\topw$ and $\topu$.

\begin{prp}
	Let $\left(s_n\right)_{n\in\setN}$ be a convergent sequence in $\setF$. Then for every $q\in\setQ$, the real sequence $\left(s_n[q]\right)_{n\in\setN}$ converges in $\setR$ in the standard topology. Conversely, if $(s_n)$ is regular and if for every $q\in\setQ$, $(s_n[q])$ converges in $\setR$ in the standard topology then $(s_n)$ converges in $\setF$.
\end{prp}

\begin{proof}
First suppose that $(s_n)$ converges in $\left(\setF, \topw\right)$. We may assume without loss of generality that $s_n\rightarrow 0$ because if $s_n\rightarrow s\neq 0$ then we can define the new sequence $s'_n = s_n-s$ and for any $q\in\setQ$, $(s'_n[q])$ will converge to $0$ as a real sequence if and only if $(s_n[q])$ converges to $s[q]$ in $\setR$. Thus, for any $\epsilon>0$ in $\setR$, there exists $N\in\setN$ such that if $n\geq N$ then
$\genval{s_n}{\topw}{\mu(\epsilon)} < \epsilon$.

Now let $q_0\in\setQ$ be given. We will show that $s_n[q_0]\rightarrow 0$ in $\setR$. So let $\epsilon>0$ in $\setR$ be given and let $n_0\in\setN$ be large enough so that $q_0\in\Gamn{n_0}$ where $\Gam$ is any finite well-bounded partition of $\setQ$. Let $\epsilon_0 = \min\left\{\epsilon, \frac{1}{n_0}\right\}$. Since $s_n\rightarrow 0$ in $\left(\setF, \topw=\topg\right)$, there exists $N\in\setN$ such that if $n\geq N$ then $\gamval{s_n}{\mu(\epsilon_0)} < \epsilon_0$.
Recall that $\gamval{s_n}{\mu(\epsilon_0)} = \max\left\{\abs{s_n[q]}\middle|q\in\Gamn{\mu(\epsilon_0)}\right\}$
and $q_0 \in \Gamn{n_0}\subseteq\Gamn{\mu(\epsilon_0)}$. It follows that if $n\geq N$ then
$\abs{s_n[q_0]} < \epsilon_0 \leq \epsilon$.
Thus, we have shown that for any $\epsilon>0$ in $\setR$ there is a $N\in\setN$ such that if $n\geq N$ then
$\abs{s_n[q_0]}<\epsilon$, and hence
$s_n[q_0]\rightarrow 0$ in $\setR$. Since the argument holds for an arbitrary $q_0\in\setQ$, it follows that, for every $q\in\setQ$, $(s_n[q])$ converges in $\setR$.

Now suppose we have a regular sequence $\left(s_n\right)_{n\in\setN}$ in $\setF$ such that for every $q\in\setQ$, $s_n[q]\rightarrow 0$ in $\setR$. We will show that $s_n\rightarrow 0$ in $\left(\setF, \topw\right)$. So let $\epsilon>0$ in $\setR$ and let $\Gam$ be any finite well-bounded partition of $\setQ$ (recall that $\topg = \topw$). We know that, for every $q\in\Gamn{\mu(\epsilon)}\subset\setQ$, there exists $N_q\in\setN$ such that if $n\geq N_q$ then $\abs{s_n[q]}<\epsilon$. Let
\[
	N= \max\left\{N_q\middle|q\in\Gamn{\mu(\epsilon)}\right\}
\]
which must exist because $\Gam$ is a finite partition. We now observe that if $n\geq N$ then, for every $q\in\Gamn{\mu(\epsilon)}$, we have that $\abs{s_n[q]}<\epsilon$
which implies that
$\gamval{s_n}{\mu(\epsilon)} < \epsilon$.
This holds for any $\epsilon>0$ in $\setR$ and hence $s_n\rightarrow 0$ in $\left(\setF, \topg = \topw\right)$.
\end{proof}

The following proposition justifies our choice to refer to $\topu$ as the ``locally uniform support topology''. As we will see, for a sequence to converge in this topology it is not sufficient that the sequence of values at each support point converges as a real sequence but rather the convergence must be locally uniform. That is, how quickly the sequence of values at one support point converges tells us something about how quickly the sequence of values at other nearby support points converge.

\begin{prp}
 Let $\left(s_n\right)_{n\in\setN}$ be a sequence in $\setF$ that converges to $s$. Then for every $\epsilon>0$ in $\setR$ and for every $q\in\setQ$ there exist $N_q\in\setN$ and $\delta_q\in\setQ$ such that if $n\geq N_q$ and $\abs{q'-q}<\delta_q$ then $\abs{(s_n-s)[q']}<\epsilon$. Conversely, if $\left(s_n\right)_{n\in\setN}$ is a regular sequence in $\setF$ and if there exists $s\in\setF$ such that for every $\epsilon>0$ in $\setR$ and for every $q\in\setQ$ there exist $N_q\in\setN$ and $\delta_q\in\setQ$ such that if $n\geq N_q$ and $\abs{q'-q}<\delta_q$ then $\abs{(s_n-s)[q']}<\epsilon$, then $(s_n)$ converges to $s$ in $\setF$
\end{prp}

\begin{proof}
	First we suppose that $(s_n)_{n\in\setN}$ converges in $\left(\setF, \topu\right)$. As in previous proofs, we may assume without loss of generality that the sequence converges to 0. So let $qin\setQ$ and $\epsilon>0$ in $\setR$ be given. Let
	\[
	\epsilon_0 = \min\left\{\epsilon, \frac{1}{\abs{q}+1}\right\}.
	\]
	Since $(s_n)$ converges in $\left(\setF, \topu\right)$, there exists $N_0\in\setN$ such that for every $n\geq N_0$, we have that
	\[
		\unval{s_n}{\frac{1}{\epsilon_0}}<\epsilon_0.
	\]
By our choice of $\epsilon_0$, it follows that
	\[
		\left\{q'\in\setQ\middle|\abs{q'-q}<1\right\} \subseteq \left\{q'\in\setQ\middle|q'<\frac{1}{\epsilon_0}\right\}.
	\]
	Thus if $\abs{q'-q}<1$ then for all $n\geq N$,
	\[
		\abs{s_n[q']}<\epsilon_0\leq \epsilon.
	\]
	
	Now we suppose that for every $\epsilon>0$ in $\setR$ and for every $q\in\setQ$ there exist $N_q\in\setN$ and $\delta_q\in\setQ$ such that if $n\geq N$ and $\abs{q'-q}<\delta_q$ then $\abs{(s_n)[q']}<\epsilon$. We wish to show that  $(s_n)_{n\in\setN}$ converges to $0$ in $\left(\setF, \topu\right)$. Thus, for every $q\in\setQ$ let $I_q= \left(q-\delta_q, q+\delta_q\right)$. Since $(s_n)$ is a regular sequence
	\[
	q_{\min}:= \min\left\{\bigcup\limits_{i=0}^{\infty}\supp\left(s_n\right)\right\}
	\]
	exists. So for any $q\in\setQ\setminus\left[q_{\min}, \frac{1}{\epsilon}\right]$ either $s_n[q] = 0$ for all $n\in\setN$ (if $q<q_{\min}$) or $\unval{\cdot}{\frac{1}{\epsilon}}$ is not dependent on the value of $s_n[q]$ (if $q>\frac{1}{\epsilon}$). Clearly $\left\{I_q\middle|q\in\left[q_{\min}, \frac{1}{\epsilon}\right]\right\}$ is an open cover of $\left[q_{\min}, \frac{1}{\epsilon}\right]$ and since every closed bounded subset of $\setQ$ is compact there must be a finite collection of rational numbers
$\left\{q_j\right\}_{j\in \{1,\ldots,J\}} \in \left[q_{\min}, 1/\epsilon\right]$
	such that
	\[
	 \left[q_{\min}, \frac{1}{\epsilon}\right]\subset \bigcup\limits_{j=1}^{J}I_{q_j}.
	\]
	Let $N:= \max\left\{N_{q_j}\middle|j\in\left\{1, \ldots, J\right\}\right\}$.
	Then, for every $q\in\left[q_{\min}, \frac{1}{\epsilon}\right]$, if $n\geq N$ then $\val{s_n[q]}<\epsilon$. But
	\[
		\unval{s_n}{\frac{1}{\epsilon}} = \sup\left\{\abs{s_n[q]}\middle|q\in\left[q_{\min}, \frac{1}{\epsilon}\right]\right\}.
	\]
It follows that if $n\geq N$ then $\unval{s_n}{\frac{1}{\epsilon}}<\epsilon$.
This argument holds for any $\epsilon>0$ in $\setR$ and hence $(s_n)$ converges to $0$ in $\left(\setF, \topu\right)$.
\end{proof}

We conclude this section with the following proposition which offers some insight into the relation between weak topologies and convergence in the Hahn field and the analogous concepts on the Levi-Civita field.

\begin{prp}
	$\topw\vert_{\setRscript} \subsetneq \topu\vert_{\setRscript}$.
\end{prp}

\begin{proof}
	First we show that $\topw\vert_{\setRscript}\subset \topu\vert_{\setRscript}$. So let $O\in\topw\vert_{\setRscript}$ be given and let $\Gam$ be any finite well-bounded partition of $\setQ$. Let $x\in O$ be given. Then there exists $r\in\setQ$ such that
$\pbg{x}{r}\cap\setRscript\subset O$.
Since $\Gam$ is well-bounded $\Gamn{\mu(r)}$ must have a maximal element which we will denote by $q_{\max}$. Let
\begin{equation}\label{eqeps}
\epsilon=\left\{\begin{array}{ll}\min\left\{r, \frac{1}{\left|q_{\max}\right|}\right\}&\mbox{if }q_{\max}\ne 0\\
r&\mbox{if }q_{\max}=0. \end{array}\right.
\end{equation}
Then $\epsilon>0$ in $\setR$.
We claim that
$\pbu{x}{\epsilon}\cap\setRscript\subset\pbg{x}{r}\cap\setRscript$.

	To prove the claim, it is sufficient to show that
$\pbu{x}{\epsilon}\subset\pbg{x}{r}$.
	So let $y\in\pbu{x}{\epsilon}$. Then we have that $\unval{y-x}{\frac{1}{\epsilon}}<\epsilon$, that is,
	\[
		\sup\left\{\abs{(y-x)[q]}\middle|q\in\left(-\infty, \frac{1}{\epsilon}\right]\cap\setQ\right\}<\epsilon.
	\]
Using Equation (\ref{eqeps}), we have that $\frac{1}{\epsilon}\geq \left|q_{\max}\right|\geq q_{\max}$, so $\Gamn{\mu(r)}\subset (-\infty, \frac{1}{\epsilon}]\cap\setQ$ and hence
	\[
		\sup\left\{\abs{(y-x)[q]}\middle|q\in\Gamn{\mu(r)}\right\}<\epsilon \leq r.
	\]
	But the left hand side is, by definition, $\gamval{y-x}{\mu(r)}$. Thus,
$\gamval{y-x}{\mu(r)}<r$, and hence $y\in\pbg{x}{r}$. This is true for any $y\in\pbu{x}{\epsilon}$. Thus, $\pbu{x}{\epsilon}\subset\pbg{x}{r}$ and hence $\pbu{x}{\epsilon}\cap\setRscript\subset\pbg{x}{r}\cap\setRscript\subset O$, as claimed. It follows that $O\in\topu\vert_{\setRscript}$ and hence $\topw\vert_{\setRscript}\subset\topu\vert_{\setRscript}$.
	
	It remains to show that $\topu\vert_{\setRscript}\not\subset\topw\vert_{\setRscript}$. We already know that $\left\{\pbg{0}{q}\middle|q\in\setQ^{+}\right\}$ is a local base of $\topw$ at $0$ so it is enough to show that for every $q\in\setQ^{+}$ there exists $x\in\pbg{0}{q}\cap\setRscript$ such that $x\not\in\pbu{0}{1}$. So let $q\in\setQ^{+}$ be given. Since $\Gam$ is a finite partition of $\setQ$, we have that $(-\infty, 1]\cap(\setQ\setminus\Gamn{\mu(q)})\neq\emptyset$ and hence we can select $s\in(-\infty, 1]\cap(\setQ\setminus\Gamn{\mu(q)})$. Let $x\in\setRscript$ be given by
$x= 2d^{s}$. Then
	\[
		\gamval{x-0}{\mu(q)} = 0 <q
	\]
	and hence $x\in\pbg{0}{q}\cap\setRscript$; but
	\[
		\unval{x-0}{1} \geq 2 > 1
	\]
and hence $x\not\in\pbu{0}{1}\cap\setRscript$. Since our choice of $q\in\setQ^{+}$ was arbitrary, it follows that
$\pbu{0}{1}\vert_{\setRscript}\not\in\topw\vert_{\setRscript}$ and hence
 $\topu\vert_{\setRscript} \not\subset \topw\vert_{\setRscript}$.
\end{proof}

	\section{Convergence of power series}
In this final section of the paper we will show that, under the weak topology, power series on the Hahn field have the same convergence criterion as that for convergence of power series on the Levi-Civita field \cite{shamberz,rstopology09}. We begin by recalling what we mean by convergence of a power series.

\begin{dfn}
	Let $(a_n)_{n\in\setN}$ be a sequence in $\setF$, and let $x_0$ be a fixed point and $x$ a variable point in $\setF$. Then we say that the power series
	\[
		\sum\limits_{n=0}^{\infty}a_n(x-x_0)^n
	\]
	converges weakly in $\setF$ if the sequence of partial sums
	\[
		S_m(x) :=  \sum\limits_{n=0}^{m}a_n(x-x_0)^n
	\]
	converges in $\left(\setF, \topw\right)$.
\end{dfn}

The following proposition provides a criterion for convergence of power series in $\left(\setF, \topw\right)$.

\begin{thm}\label{scaledwpscc}
	Let $(a_n)_{n\in\setN}$ be a regular sequence in $\setF$, let $S=\bigcup\limits_{n=0}^{\infty}\supp(a_n)$, and assume that
	\[
		-\liminf\limits_{n\rightarrow\infty}\left(\frac{\lambda(a_n)}{n}\right) = \limsup\limits_{n\rightarrow\infty}\left(-\frac{\lambda(a_n)}{n}\right) = 0.
	\]
	Let
	\[
		r= \frac{1}{\sup\left\{\limsup\limits_{n\rightarrow\infty}\abs{a_n[q]}^{\frac{1}{n}}:q\in S\right\}}
	\]
	and let $x\in\setF$ be such that $\lambda(x)\geq 0$. Then the power series
	\[
		\sum\limits_{n=0}^{\infty}a_nx^n
	\]
	converges absolutely in $\left(\setF, \topw\right)$ if $\abs{x[0]}<r$ and diverges in $\left(\setF, \topw\right)$ if $\abs{x[0]}>r$.
\end{thm}

\begin{proof}
	If $\lambda(x)>0$ then $\sum\limits_{n=0}^{\infty}\left|a_nx^n\right|$ converges in $\left(\setF, \topv\right)$ and hence in $\left(\setF, \topw\right)$ \cite[Theorem 2.13]{graftonthesis}; so it remains to consider the case where $\lambda(x)=0$. First assume that $\abs{x[0]}<r$; we will show that, for every $q\in S$, the real power series
	\begin{equation}\label{series1}
		\sum\limits_{n=0}^{\infty}\left|a_n[q](x[0])^n\right|
	\end{equation}
	converges in $\setR$. So let $q_0\in S$ be given. Since $\abs{x[0]}<r$ we have that
	\[
		\abs{x[0]}<\frac{1}{\sup\left\{\limsup\limits_{n\rightarrow\infty}\abs{a_n[q]}^{\frac{1}{n}}:q\in S\right\}}
	\le\frac{1}{\limsup\limits_{n\rightarrow\infty}\abs{a_n[q_0]}^{\frac{1}{n}}}.
	\]
It follows that
$\limsup\limits_{n\rightarrow\infty}\abs{a_n[q_0](x[0])^n}^{\frac{1}{n}}<1$.
Let $c\in\setR$ be such that
	\[
		\limsup\limits_{n\rightarrow\infty}\abs{a_n[q_0](x[0])^n}^{\frac{1}{n}}< c < 1.
	\]
Then
$\sum\limits_{n=0}^{\infty}c^n$ converges to $1/(1-c)$ in $\setR$.
Since $\limsup\limits_{n\rightarrow\infty}\abs{a_n[q_0](x[0])^n}^{\frac{1}{n}}< c$, there exists $N\in\setN$ such that
	\[
		\abs{a_n[q_0](x[0])^n}^{\frac{1}{n}} < c
	\]
	 for all $n\geq N$. Thus, using the comparison test, it follows that
$\sum\limits_{n=0}^{\infty}\left|a_n[q_0](x[0])^n\right|$ converges in $\setR$. Since our choice of $q_0\in S$ was arbitrary, we conclude that the power series in Equation (\ref{series1}) converges in $\setR$ for every $q\in S$.

Next we claim that for all $q\in S$,
	\[
		\sum\limits_{n=0}^{\infty}\left|a_n[q]x^n\right|
	\]
	converges in $\left(\setF,\topw\right)$. So let $q_0\in S$ be given and consider the sequence of partial sums $\left(S_m\right)_{m\in\setN}$, where for each $m\in\setN$,
	$S_m = \sum\limits_{n=0}^{m} \left|a_n[q_0]x^n\right|$.
We know already that $(S_m)$ is a regular sequence because $\lambda(x)=0$ \cite[Theorem 2.3, Corollary 2.12.1]{graftonthesis}, so it remains to show that for any $t\in\setQ$, the real sequence $(S_m[t])$ converges in $\setR$. So let $t\in\setQ$ and $\epsilon>0$ in $\setR$ be given. We will show that there exists $N\in\setN$ such that if $m_2>m_1>N$ then
\[
\left|S_{m_2}[t]-S_{m_1}[t]\right|=\sum\limits_{n=m_1}^{m_2}\abs{a_n[q_0]x^n[t]}<\epsilon,
\]
thus showing that $(S_m[t])$ is a cauchy sequence and hence  convergent in $\setR$.

Let $h= x-x[0]$ and let $N'\in\setN$ be such that $N'\lambda(h)>t$. We have, for any $n\in\setN$, that
	\begin{align*}
		\left(\left(x[0] + h\right)^n\right)[t] &= \left(\sum\limits_{l=0}^{n}\frac{n!}{(n-l)!l!}h^l(x[0])^{n-l}\right)[t]\\
		&= \sum\limits_{l=0}^{\min\{N',n\}}\frac{n!}{(n-l)!l!}h^l[t](x[0])^{n-l}.
	\end{align*}
It follows that, for $m_2>m_1>N'$, we have that
	\begin{align*}
		\sum\limits_{n=m_1}^{m_2}\abs{a_n[q_0]x^n[t]}&=\sum\limits_{n=m_1}^{m_2}\abs{a_n[q_0](x[0]+h)^n[t]} \\
&= \sum\limits_{n=m_1}^{m_2}\abs{a_n[q_0]}\abs{ \sum\limits_{l=0}^{N'}\frac{n!}{(n-l)!l!}h^l[t]x[0]^{n-l}}\\
		&\leq \sum\limits_{n=m_1}^{m_2}\sum\limits_{l=0}^{N'}\abs{a_n[q_0]}\abs{h[t]^l}\abs{x[0]}^{n-l}\frac{n!}{(n-l)!l!}\\
		&\leq \left(\sum\limits_{l=0}^{N'}\frac{\abs{h[t]^l}\abs{x[0]}^{N'-l}}{l!}\right)\left( \sum\limits_{n=m_1}^{m_2}\abs{a_n[q_0]}n^{N'}\abs{x[0]}^{n-N'}\right).
	\end{align*}
	The first term in the final expression above is independent of $m_1$ and $m_2$; moreover, since
	\[
		\abs{x[0]} < r
		\leq \frac{1}{\limsup\limits_{n\rightarrow\infty}\left\{\abs{a_n[q_0]}^{\frac{1}{n}}\right\}}
		=  \frac{1}{\limsup\limits_{n\rightarrow\infty}\left\{\abs{a_n[q_0]n^{N'}}^{\frac{1}{n}}\right\}},
	\]
 the real series
	\[
		 \sum\limits_{n=0}^{\infty}\abs{a_n[q_0]}n^{N'}\abs{x[0]}^{n-N'}
	\]
	must converge in $\setR$. Thus, there exists $N''\in\setN$ such that if $m_2>m_1>N''$ then
	\[
	\sum\limits_{n=m_1}^{m_2}\abs{a_n[q_0]}n^{N'}\abs{x[0]}^{n-N'}
		 <\frac{\epsilon}{\sum\limits_{l=0}^{N'}\frac{\abs{h[t]^l}\abs{x[0]}^{N'-l}}{l!}}.
	\]
	It follows that if $m_2>m_1>\max\left\{N', N''\right\}$ then
	\begin{align*}
		\sum\limits_{n=m_1}^{m_2}\abs{a_n[q_0](x[0]+h)^n[t]} &\leq \left(\sum\limits_{l=0}^{N'}\frac{\abs{h[t]^l}\abs{x[0]}^{N'-l}}{l!}\right)\left( \sum\limits_{n=m_1}^{m_2}\abs{a_n[q_0]}n^{N'}\abs{x[0]}^{n-N'}\right)\\
			&< \epsilon.
	\end{align*}
 Since $q_0\in\setQ$ was arbitrary, it follows that $\sum\limits_{n=0}^{\infty}\left|a_n[q]x^n\right|$ converges in $\left(\setF,\topw\right)$, for every $q\in\setQ$.

 Finally we show that $\sum\limits_{n=0}^{\infty}\left|a_nx^n\right|$
converges in $\left(\setF,\topw\right)$. We have already shown that, for every $q\in\setQ$, $\sum\limits_{n=0}^{\infty}\left|a_n[q]x^n\right|$ converges in $\left(\setF,\topw\right)$. Moreover, $\lambda(\sum\limits_{n=0}^{\infty}a_n[q]x^n)\geq 0$ and hence
	\[
		\sum\limits_{q\in S}d^q\sum\limits_{n=0}^{\infty}\left|a_n[q]x^n\right|
	\]
has a well ordered support. Let $t\in\setQ$ be given. Then
	\begin{align*}
		\left(\sum\limits_{q\in S}d^q\sum\limits_{n=0}^{\infty}\left|a_n[q]x^n\right|\right)[t] &= \sum\limits_{q\in S}\left(d^q\sum\limits_{n=0}^{\infty}\left|a_n[q]x^n\right|\right)[t]\\
		&=  \sum\limits_{q\in S}\left(\sum\limits_{t_1+t_2 = t}d^q[t_1]\left(\sum\limits_{n=0}^{\infty}\left|a_n[q]x^n\right|\right)[t_2]\right)\\
		&=  \sum\limits_{q\in S}d^q[q]\left(\sum\limits_{n=0}^{\infty}\left|a_n[q]x^n\right|\right)[t-q]\\
		&=  \sum\limits_{q\in S}\sum\limits_{n=0}^{\infty}\left|a_n[q]\right|\left|x^n\right|[t-q].
	\end{align*}
	Now, let $S_0\subset S$ be the set of all $q\in S$ such that $t-q\in\bigcup\limits_{n=0}^{\infty}\supp(x^n)$. Since $S$ and $\bigcup\limits_{n=0}^{\infty}\supp(x^n)$ are both well-ordered, it follows that $S_0$ is finite \cite[Theorem 1.3]{graftonthesis}. It follows that
\begin{align*}
		\left(\sum\limits_{q\in S}d^q\sum\limits_{n=0}^{\infty}\left|a_n[q]x^n\right|\right)[t]&=\sum\limits_{q\in S}\sum\limits_{n=0}^{\infty}\left|a_n[q]\right|\left|x^n\right|[t-q]\\
= &\sum\limits_{q\in S_0}\left(\sum\limits_{n=0}^{\infty}\left|a_n[q]x^n\right|\right)[t-q]
\end{align*}
is finite. Thus, $\sum\limits_{q\in S}d^q\sum\limits_{n=0}^{\infty}\left|a_n[q]x^n\right|$ converges in $(\setF,\topw)$. Moreover, we have that

	\begin{align*}
		&\left(\sum\limits_{q\in S}d^q\sum\limits_{n=0}^{\infty}\left|a_n[q]x^n\right|\right)[t]=\sum\limits_{q\in S_0}d^q[q]\sum\limits_{n=0}^{\infty}\left|a_n[q]\right|\left|x^n\right|[t-q]\\
		= &\sum\limits_{n=0}^{\infty}\sum\limits_{q\in S_0}d^q[q]\left|a_n[q]\right|\left|x^n\right|[t-q]
		=\sum\limits_{n=0}^{\infty}\sum\limits_{q\in S}d^q[q]\left|a_n[q]\right|\left|x^n\right|[t-q]\\
		=&\sum\limits_{n=0}^{\infty}\sum\limits_{q\in S}\left|a_n[q]\right|\sum\limits_{t_1+t_2=t}d^q[t_1]\left|x^n\right|[t_2]
		=\left(\sum\limits_{n=0}^{\infty}\sum\limits_{q\in S}|a_n[q]|d^q|x^n|\right)[t]\\
		=&\left(\sum\limits_{n=0}^{\infty}\left(\sum\limits_{q\in S}|a_n[q]|d^q\right)|x^n|\right)[t]
		\ge\left(\sum\limits_{n=0}^{\infty}\left|\sum\limits_{q\in S}a_n[q]d^q\right||x^n|\right)[t]\\
=&\left(\sum\limits_{n=0}^{\infty}|a_nx^n|\right)[t].
	\end{align*}
	So for every $t\in\setQ$, $\left((\sum\limits_{n=0}^{m}\left|a_nx^n\right|)[t]\right)_{m\in\setN}$ converges as a real sequence and, moreover, $\left(\sum\limits_{n=0}^{m}|a_nx^n|\right)_{m\in\setN}$ is a regular sequence because $(a_n)_{n\in\setN}$ is regular and $\lambda(x)\geq 0$. Hence
$\sum\limits_{n=0}^{m}\left|a_nx^n\right|$ converges in $(\setF,\topw)$.

Now let $x\in\setF$ be such $|x[0]|>r$; we will show that $\sum\limits_{n=0}^{\infty}a_nx^n$ diverges in $\left(\setF,\topw\right)$. Assume to the contrary that $\sum\limits_{n=0}^{\infty}a_nx^n$ converges in $\left(\setF,\topw\right)$. Let $h=x-x[0]$. Then, since  $\sum\limits_{n=0}^{\infty}a_nx^n$ converges in $\left(\setF,\topw\right)$, we have that
	\begin{align*}
		\sum\limits_{n=0}^{\infty}a_nx^n &= \sum\limits_{n=0}^{\infty}a_n(x[0]+h)^n
		=\sum\limits_{n=0}^{\infty}a_n\left(\sum\limits_{k=0}^{n}\frac{n!}{k!(n-k)!}x[0]^{n-k}h^k\right)\\
		&=\sum\limits_{k=0}^{\infty}\sum\limits_{n=k}^{\infty}a_n\frac{n!}{k!(n-k)!}x[0]^{n-k}h^k
		= \sum\limits_{k=0}^{\infty}\frac{\left(\sum\limits_{n=k}^{\infty}\frac{a_nn!}{(n-k)!}x[0]^n\right)}{k!}h^k\\
		&= \sum\limits_{n=0}^{\infty}a_nx[0]^n + \sum\limits_{k=1}^{\infty}\frac{\left(\sum\limits_{n=k}^{\infty}\frac{a_nn!}{(n-k)!}x[0]^n\right)}{k!}h^k.
	\end{align*}
	Observe that, for every $k\in\setN$,
	\[
		\limsup\limits_{n\rightarrow\infty}\left\{\abs{\frac{n!}{(n-k)!}}^{\frac{1}{n}}\right\} = 1.
	\]
	Thus, for every $k\in\setN$ and $q\in\setQ$, we have that
	\[
		\limsup\limits_{n\rightarrow\infty}\left\{\abs{a_n[q]\frac{n!}{(n-k)!}}^{\frac{1}{n}}\right\} = \limsup\limits_{n\rightarrow\infty}\left\{\abs{a_n[q]}^{\frac{1}{n}}\right\}.
	\]
	It follows that, for every $q\in\setQ$,
	\[
		\sum\limits_{n=k}^{\infty}\frac{a_n[q]n!}{(n-k)!}x[0]^n
	\]
	diverges in $\setR$ only when
	\[
		\sum\limits_{n=k}^{\infty}a_n[q]x[0]^n
	\]
	diverges in $\setR$. Since $\abs{x[0]}>r$, we have by definition of $r$ that
	\[
		\frac{1}{\abs{x[0]}} < \sup\left\{\limsup\limits_{n\rightarrow\infty}\abs{a_n[q]}^{\frac{1}{n}}:q\in S\right\}.
	\]
	Therefore there is at least one $q\in S$ such that
	\[
		\frac{1}{\abs{x[0]}} < \limsup\limits_{n\rightarrow\infty}\abs{a_n[q]}^{\frac{1}{n}}
	\]
	and hence
	\[
		\abs{x[0]} > \frac{1}{\limsup\limits_{n\rightarrow \infty}\abs{a_n[q]}^{\frac{1}{n}}}.
	\]
	Thus, by the root test, we have that
	$\sum\limits_{n=0}^{\infty}a_n[q]x^n[0]$
	diverges in $\setR$. Let $q_0\in S$ be the smallest such element (which exists since $S$ is well-ordered)and let $q_1 = \lambda(h)$. Then, for any $k\geq 1$, the smallest $q\in\setQ$ such that
	\[
		\left(\frac{\left(\sum\limits_{n=k}^{\infty}\frac{a_nn!}{(n-k)!}x[0]^n\right)}{k!}h^k\right)[q]
	\]
	diverges is $q = q_0 + kq_1$. Since $kq_1>0$ we therefore have that
	\[
		\left(\sum\limits_{k=1}^{\infty}\frac{\left(\sum\limits_{n=k}^{\infty}\frac{a_nn!}{(n-k)!}x[0]^n\right)}{k!}h^k\right)[q_0]
	\]
	must converge. However,
$\sum\limits_{n=0}^{\infty}a_n[q_0]x^n[0]$
	diverges in $\setR$ and
	\[
		\left(\sum\limits_{n=0}^{\infty}a_nx^n\right)[q_0] = \sum\limits_{n=0}^{\infty}a_n[q_0]x^n[0] + \left(\sum\limits_{k=1}^{\infty}\frac{\left(\sum\limits_{n=k}^{\infty}\frac{a_nn!}{(n-k)!}x[0]^n\right)}{k!}h^k\right)[q_0].
	\]
It follows that $\left(\sum\limits_{n=0}^{\infty}a_nx^n\right)[q_0]$ diverges in $\setR$. This contradicts the assumption that $\sum\limits_{n=0}^{\infty}a_nx^n$ converges in $\left(\setF,\topw\right)$. Hence $\sum\limits_{n=0}^{\infty}a_nx^n$ diverges in $\left(\setF,\topw\right)$.
\end{proof}

\begin{crl}
	Let $(a_n)_{n\in\setN}$ be a sequence in $\setF$ and assume that
	\[
	-\liminf\limits_{n\rightarrow\infty}\left(\frac{\lambda(a_n)}{n}\right) = \limsup\limits_{n\rightarrow\infty}\left(-\frac{\lambda(a_n)}{n}\right) = \lambda_0.
	\]
	 Let
\[
		r= \frac{1}{\sup\left\{\limsup\limits_{n\rightarrow\infty}\abs{a_n[q]}^{\frac{1}{n}}:q\in\bigcup\limits_{n\in\setN}\supp(a_n)\right\}},
	\]
let $x_0\in\setF$ be fixed, and let $x\in\setF$ be such that $\lambda(x-x_0)\geq \lambda_0$. Finally, assume that $(a_nd^{n\lambda_0})_{n\in\setN}$ is a regular sequence. Then $\sum\limits_{n=0}^{\infty}a_n(x-x_0)^n$
	converges absolutely in $\left(\setF, \topw\right)$ if $\abs{(x-x_0)[\lambda_0]}<r$ and diverges in $\left(\setF, \topw\right)$ if $\abs{(x-x_0)[\lambda_0]}>r$.
	\end{crl}

\begin{proof}
	For every $n\in\setN$ let $b_n = a_nd^{n\lambda_0}$; and let $y = d^{-\lambda_0}(x-x_0)$. Then
	\begin{align*}
		\limsup\limits_{n\rightarrow\infty}\left(-\frac{\lambda(b_n)}{n}\right) &= \limsup\limits_{n\rightarrow\infty}\left(-\frac{\lambda(a_nd^{n\lambda_0})}{n}\right)
		= \limsup\limits_{n\rightarrow\infty}\left(-\frac{\lambda(a_n)}{n} - \frac{n\lambda_0}{n}\right)\\
		&= \limsup\limits_{n\rightarrow\infty}\left(-\frac{\lambda(a_n)}{n}\right)  - \lambda_0
		=\lambda_0-\lambda_0 = 0.
	\end{align*}
	Moreover, for every $n\in\setN$, we have that
	\[
		b_ny^n = a_nd^{n\lambda_0}\left(d^{-\lambda_0}(x-x_0)\right)^n= a_n(x-x_0)^n,
	\]
and hence $\sum\limits_{n=0}^{\infty}a_n(x-x_0)^n  = \sum\limits_{n=0}^{\infty}b_ny^n$.
Finally note that
	\begin{align*}
		\frac{1}{r} &= \sup\left\{\limsup\limits_{n\rightarrow\infty}\abs{a_n[q]}^{\frac{1}{n}}:q\in\bigcup\limits_{n\in\setN}\supp(a_n)\right\}\\
		&=  \sup\left\{\limsup\limits_{n\rightarrow\infty}\abs{(a_nd^{n\lambda_0})[q+n\lambda_0]}^{\frac{1}{n}}:q\in\bigcup\limits_{n\in\setN}\supp(a_n)\right\}\\
		&=  \sup\left\{\limsup\limits_{n\rightarrow\infty}\abs{b_n[t]}^{\frac{1}{n}}:t\in\bigcup\limits_{n\in\setN}\supp(b_n)\right\}.
	\end{align*}

Since $(b_n)_{n\in\setN}$ is a regular sequence in $\setF$ with
	\[
		-\liminf\limits_{n\rightarrow\infty}\left(\frac{\lambda(b_n)}{n}\right) = \limsup\limits_{n\rightarrow\infty}\left(-\frac{\lambda(b_n)}{n}\right) = 0
	\]
and since $\lambda(y)=-\lambda_0+\lambda(x-x_0)\geq 0$, it follows immediately from Theorem \ref{scaledwpscc} that
$\sum\limits_{n=0}^{\infty}a_n(x-x_0)^n=\sum\limits_{n=0}^{\infty}b_ny^n$
converges absolutely in  $\left(\setF, \topw\right)$ if $\abs{y[0]}<r$ and diverges in  $\left(\setF, \topw\right)$ if $\abs{y[0]}>r$. However, $\abs{y[0]}<r$ if and only if $\abs{(x-x_0)[\lambda_0]}<r$ and  $\abs{y[0]}>r$ if and only if  $\abs{(x-x_0)[\lambda_0]}>r$. Thus, the proposition is proved.
\end{proof}

%
			
%
\end{document}